\documentclass{amsart}

\usepackage{amsmath, amssymb, latexsym}
\usepackage{epsfig}

\usepackage{tikz}

\usepackage[colorlinks=true,citecolor=blue]{hyperref}
\usepackage{latexsym,amsmath,amsthm,amssymb,amsfonts,amscd,multirow}
\usepackage{epsfig,subfigure,verbatim,graphics}
\usepackage{xcolor}

\usepackage{makecell}

\newtheorem{theorem}{Theorem}[section]

\newtheorem{remark}[theorem]{Remark}

\theoremstyle{definition}
\newtheorem{definition}[theorem]{Definition}
\newtheorem{example}[theorem]{Example}

\begin{document}

\title[Instability of Betti Sequence]
{Instability of the Betti Sequence for Persistent Homology and a Stabilized Version of the Betti Sequence}

\author[Johnson]{Megan Johnson}
\address{Department of Mathematics\\
 University at Buffalo, The State University of New York, Buffalo, NY, 14260\\ USA }
\email{meganjoh@buffalo.edu}
%\thanks{The first author's work was supported in part by KSIAM}

\author[Jung]{Jae-Hun Jung}
%\author[Third]{J.KSIAM Third\authormark{2 \dag}}
\address{Department of Mathematics
  \& POSTECH Mathematical Institute for Data Science (MINDS)\\ Pohang University of Science and Technology, Pohang 37673 \\Korea}
%\email{\{jksiam-2,jksiam-3\}@second.ksiam.org}
\email{jung153@postech.ac.kr}
%\thanks{\authormark{2} Corresponding author.}


\begin{abstract}
Topological Data Analysis (TDA), a relatively new field of data analysis, has proved very useful in a variety of applications. The main persistence tool from TDA is persistent homology in which data structure is examined at many scales. Representations of persistent homology include persistence barcodes and persistence diagrams, both of which are not straightforward to reconcile with traditional machine learning algorithms as they are sets of intervals or multisets. The problem of faithfully representing barcodes and persistent diagrams has been pursued along two main avenues: kernel methods and vectorizations. One vectorization is the Betti sequence, or Betti curve, derived from the persistence barcode. While the Betti sequence has been used in classification problems in various applications, to our knowledge, the stability of the sequence has never before been discussed. In this paper we show that the Betti sequence is unstable under the 1-Wasserstein metric with regards to small perturbations in the barcode from which it is calculated. In addition, we propose a novel stabilized version of the Betti sequence based on the Gaussian smoothing seen in the Stable Persistence Bag of Words for persistent homology. We then introduce the normalized cumulative Betti sequence and provide numerical examples that support the main statement of the paper. 
\end{abstract}

\maketitle

\section{Introduction}
\label{intro}
%
%%%% Intro
Topological Data Analysis (TDA) is rising field useful for the analysis of high-dimensional data structure \cite{carlsson:2009}. The popularly used tool in TDA is persistent homology, introduced by Edelsbrunner et al. in 2002 \cite{Edelsbrunner2002}, where snapshots of the topological structure of the data set are taken at many different scales and the results are compared from one scale to the next. Singular homology is, in general, hard to compute, and almost impossible to compute for real-world data but when the given topological space is approximated from finitely many points, simplicial homology, can instead be used as it is easily computable. Persistent homology is obtained by computing simplicial homology at different scales. 
%%%% Previous Applications
Previous applications of TDA and persistent homology include 3D shape segmentation \cite{carriere:3Dshapes}, astrophysics \cite{cole_shiu,heydenreich2020,Xu2019}, biology and medicines \cite{McGuirl2020,Nicponski2020}, and neuroscience \cite{bendich2016, Sizemore:2019} to name a few.
%%%% Overview and Motivation

Popular representations of persistent homology include persistent diagrams and persistent barcodes. Persistent barcodes and persistent diagrams are mathematically equivalent and they demonstrate how the homological structures of the given data change according to scale. Although these are useful for data analysis, they are not necessarily compatible with typical machine learning workflows in their raw forms as they are designed as multisets and collections of intervals, respectively. One way to reconcile these persistent homology representations and machine learning algorithms is to vectorize the persistent diagrams or the persistent barcodes \cite{adams2015persistence}. There are several vectorization methods including the topological vector, persistence vector, and persistence images. These vectorization methods are, in general, more computationally efficient than kernel methods. The construction of the topology vectors and persistence vectors are straightforward and easy to implement. The persistence image is less straightforward (i.e. slower) to compute but delivers better classification results, in general.  Another vectorization that we consider in this paper is the Betti sequence \cite{Umeda2017} which contains the Betti numbers of the homology groups of the simplicial complex built on the data set at all scales of the persistent homology. 
%%%% Contributions

In this paper, we prove by example that the Betti sequence is unstable with respect to the 1-Wasserstein distance. In other words, a small change in a persistent diagram leads to a large change in the 1-Wasserstein norm of the Betti Sequence.  To our knowledge, the instability of the Betti sequence, although mentioned in \cite{chung2020persistence}, has not yet been explicitly shown. In practice, such large change may not be significant if finer filtration intervals are chosen but to remedy the instability in the Betti sequence we propose a new stabilized version and prove its stability. In the stabilization, we adopt a similar Gaussian-smoothing approach as in \cite{johnson:2020,zielinski2018persistence}. In this paper, we show numerical examples that support our statement and show the validation of the proposed stabilization of the Betti sequence. 
%%%%  Organization of Paper

This paper will be organized as follows: Section 2 will cover the background of persistent homology and its representations and it will cover the definition and instability of the Betti sequence. Section 3 will introduce our proposed stabilization of the Betti sequence and provide a proof of its stability. Section 4 will demonstrate the normalized cumulative Betti sequence on various data sets. Section 5 contains our concluding remarks.

\section{Persistent Homology and Betti Sequence}

%\subsection{Persistent homology} 
%\label{back}
The main persistent homology representations that we consider in this paper are the persistence diagram and the persistence barcode.  Both are equivalent and each one can be recovered from the other while each can be used differently in terms of numerical manipulations. Let $X$ be a given topological space. Singular homology describes the homological structure of $X$ with the $n$-dimensional homology group $H_n$. The homology group, $H_n$ is defined by $H_n = Ker(\delta_n)/Im(\delta_{n+1})$, which is the quotient group of the kernel, $Ker(\delta_n)$, and image groups, $Im(\delta_{n+1})$, of the boundary map, $\delta_n$: $C_n(X) \rightarrow C_{n-1}(X)$ where $C_n(X)$ is the free abelian group whose basis is the set of singular $n$-simplices in $X$. Roughly speaking, the rank of the $n$-dimensional homology group $H_n$, called the $n$th Betti number, indicates how many $n$-dimensional holes are there in $X$. This information is useful in understanding the topological structure of $X$. However, computing $H_n$ is not easy as $X$ is an arbitrary topological space in general. In fact, it is not practical to use for $X$ from real-world applications. Thus, in order to use homological features of $X$ for data analysis, instead of using $X$ directly, we use a point cloud sampled from $X$.

We use typical building algorithms to obtain the point cloud approximation of $X$, known as the simplicial complex $K$. There are various ways of constructing $K$ including the Vietoris-Rips complex, $V(\tau)$ where the non-negative real number $\tau$ is known as the filtration parameter. With the given value of $\tau$, $V(\tau)$ is constructed by gluing simplices whose pairwise distance is within $\tau$. However, it is not known which value of $\tau$ approximates $X$ best. For this reason, we construct $V$ for various $\tau$, which gives us the notion of {\it persistence}. The $n$th homology, $H_n$ corresponding to $V(\tau)$ can be defined similarly as above. Let $\beta_n$ be the Betti number for $H_n$. $\beta_0$ represents the number of connected components in $K$ and $\beta_n$  the number of $n$-dimensional cycles or holes. As we have the natural inclusion of $V(\tau_i)  \hookrightarrow V(\tau_j) $, $\tau_i \le \tau_j$  and a homomorphism $H_n(V(\tau_i)) \rightarrow H_n(V(\tau_j))$, we have the relation between $\beta_n$ versus $\tau$, which generates the graph of the persistent barcodes in the considered $n$-dimension. On the persistent barcodes, an interval of filtration values corresponding to the same $\beta_n$ is known as a bar and indicates the $n$-dimensional hole structure of $K$. If we call the starting point of each bar the {\it birth} and the ending point the {\it death} we can create the persistent diagram multiset by considering all points of the form $(birth, death)$. Vectorizations of persistent diagrams and persistent barcodes and their stability are key aspects that we consider in this paper. 

%%%%%%%%%%%%%%%%%%%%%%%%%%%%%%%%%%%%%%%%%%%%%%%%%%%%%%%%%%%%%%%%%%%%%%
\subsection{Definition of the Betti Sequence}

As explained above, now we consider $K$ instead of $X$ by sampling finitely many distinct points from $X$. 
Recall the Betti number, $\beta_n$,  is the rank of the $n$th dimensional homology group $H_n(K,R)$. More specifically, $\beta_0$ represents the number of connected components in $K$, $\beta_1$ the number of $1$-dimensional holes, etc. As we have the natural inclusion of our filtered simplicial complex $\{VR_{\tau_i}\}$, namely $i: VR_{\tau_i}  \hookrightarrow VR_{\tau_j}$, $\tau_i \le \tau_j$, we have a homomorphism $H_n(VR_{\tau_i}) \rightarrow H_n(VR_{\tau_j})$. This defines a relationship between $\beta_n$ and $\tau$, which is used to generate the persistence barcode of dimension $n$. 

The Betti sequence, or Betti curve, originally defined in 2017, is the vectorization of Betti numbers obtained in persistent homology \cite{Umeda2017}. The Betti sequence uses the persistence barcode and a discretization of the filtration interval to define the vectorization. At each value of $\tau$ in the discretization, we count the number of generators existing at that filtration and that is our vector entry. We provide the formal definition of the Betti sequence below. 

\begin{definition}
Given a persistence barcode of dimension $n$ with finitely many persistence intervals and a maximum filtration $\tau_{max} > 0$, let $\{\tau_i\}_{i =1}^{M}$ be equally spaced points in $[0,\tau_{max}]$. Let $\vec{v}_b = \left(v_i\right)^{M}_1$ be the vector whose entries $v_i$ count the number of persistence intervals in the barcode existing for the filtration value $\tau_i$. \end{definition}

This definition suffers from the following flaw: persistence bars, if they fall entirely in between discretization values, are not counted at all and have no impact on the Betti sequence. Note that if the mesh size is exceedingly small then any bar which falls entirely in between the tau values is likely due to noise and their omission might not be an issue. However, if the mesh size is chosen poorly this could result in a major loss of information. 

We propose an alternate definition of the Betti sequence which agrees with the original definition in the limit as the mesh size goes to zero. 

\begin{definition}
Given a persistence barcode of dimension $n$ with finitely many bars and a maximum filtration $\tau_{max} > 0$, divide the interval $[0,\tau_{max}]$ into $M$ equal subintervals of length $\Delta \tau = \frac{\tau_{max}}{M}$. 
Let $\vec{v}_b = \left(v_i\right)^{M}_1$ be the vector whose entries $v_i$ count the number of bars in the barcode that exist for at least one filtration value $\tau$ in the $i$th subinterval $\left(\tau_{i-1} , \tau_{i}\right)$ of the filtration interval $[0,\tau_{max}]$. \end{definition}

This alternate definition has the added benefit of being easily translatable to the language of persistence diagrams making the study of the stability of the Betti sequence with respect to the p-Wasserstein metric possible. 

\subsection{Redefining the Betti Sequence via the Persistence Diagram}
In order to discuss stability with respect to the Wasserstein metric, we need to redefine the Betti sequence in terms of the persistent diagram. Consider a persistent diagram $B$ with finitely many off-diagonal points. Then a bar on the corresponding persistent barcodes exists at some filtration value $\tau$ in the subinterval $\left( \tau_{i -1 }, \tau_{i}\right)$ if and only if its birth-death point on the persistent diagram falls in the shaded region, call it $D_i$, illustrated in Figures \ref{fig:shadedregion} and \ref{fig:shadedregion1}

\begin{figure}[!htb]
\minipage{0.99\textwidth}

\begin{center}
\begin{tikzpicture}[scale=0.5]
\draw[black, thick] (0,0) -- coordinate (x axis mid) (8,0);
\draw[black, thick] (0,0) -- coordinate (y axis mid) (0,8);
\draw[black] (0,0) --(8,8);
%\draw[black] (0,0) --(4,4);

%labels      
	\node[below=0.3cm] at (x axis mid) {Birth};
	\node[rotate=90, above=0.5cm] at (y axis mid) {Death};

\filldraw[color=white, fill=gray!20, ultra thin, dashed] (0,2) -- (2,2) -- (4,4) -- (4,8) --
(0,8) --(0,2) -- cycle;

\draw[gray, thick, dashed] (0,2) -- (2,2);
\draw[gray, thick, dashed] (2,2) -- (4,4);
\draw[gray, thick, dashed] (4,4) -- (4,8);
\draw[gray, thick, dashed] (0,2) -- (0,8);

\filldraw[black] (2,0) circle (1pt) node[anchor=north] {\small $\tau_{i-1}$};
\filldraw[black] (0,2) circle (1pt) node[anchor=east] {\small $\tau_{i-1}$};
\filldraw[black] (4,0) circle (1pt) node[anchor=north] {\small $\tau_{i}$};
\filldraw[black] (0,4) circle (1pt) node[anchor=east] {\small $\tau_{i}$};

\filldraw[black] (8,8) circle (0.1pt) node[anchor = west] {\large $\Delta$};
\node at (2,4) {\large $D_i$};

\end{tikzpicture}
 \end{center}

\endminipage\hfill
%\minipage{0.49\textwidth}
%
%\begin{tikzpicture}[scale=0.6]
%\draw[black, thick] (0,0) -- coordinate (x axis mid) (4,0);
%\draw[black, thick] (0,0) -- coordinate (y axis mid) (0,4);
%\draw[black] (0,0) --(4,4);
%
%%labels      
%	\node[below=0.3cm] at (x axis mid) {Birth};
%	\node[rotate=90, above=0.5cm] at (y axis mid) {Death};
%
%
%\filldraw [color=blue!60,fill=blue!10, ultra thin, dashed] (0,1) rectangle (1,2);
%\filldraw [color=white!60,fill=orange!10, ultra thin, dashed] (0,2) rectangle (1,4);
%\filldraw [color=white!60,fill=red!10, ultra thin, dashed] (1,2) rectangle (2,4);
%\filldraw[color=green, fill=green!10, ultra thin, dashed] (1,1) -- (2,2) -- (1,2) -- cycle;
%
%\draw[color=red!60, ultra thin, dashed] (2,2) -- (2,4);
%\draw[color=red!60, ultra thin, dashed] (1,2) -- (1,4);
%\draw[color=red!60, ultra thin, dashed] (1,2) -- (2,2);
%\draw[color=orange!60, ultra thin, dashed] (1,2) -- (1,4);
%\draw[color=orange!60, ultra thin, dashed] (0,2) -- (0,4);
%\draw[color=orange!60, ultra thin, dashed] (0,2) -- (1,2);
%
%\filldraw[black] (1,0) circle (1pt) node[anchor=north] {\small $\tau_{i-1}$};
%\filldraw[black] (0,1) circle (1pt) node[anchor=east] {\small $\tau_{i-1}$};
%\filldraw[black] (2,0) circle (1pt) node[anchor=north] {\small $\tau_{i}$};
%\filldraw[black] (0,2) circle (1pt) node[anchor=east] {\small $\tau_{i}$};
%\filldraw[black] (1,2) circle (1pt) node[anchor=south] {\small ($\tau_{i-1}, \tau_{i})$};
%
%\filldraw[black] (4,4) circle (0.1pt) node[anchor = west] {\large $\Delta$};
%
%\end{tikzpicture}
%
%\endminipage\hfill
\caption{The shaded region, $D_i$, of the persistence diagram corresponding to the filtration interval $[\tau_{i-1}, \tau_i]$ on a barcode. Here $\Delta$ is the diagonal.} \label{fig:shadedregion}
\end{figure}
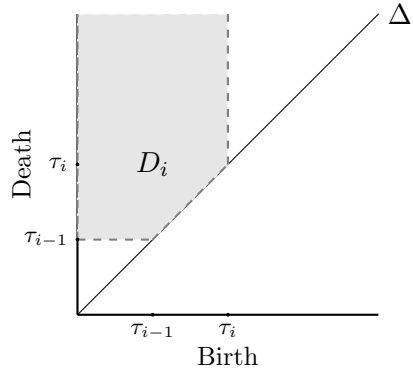

\begin{figure}[!htb]

\minipage{0.49\textwidth}

\begin{tikzpicture}[scale=0.6]
\draw[black, thick] (0,0) -- coordinate (x axis mid) (8,0);
\draw[black, thick] (0,0) -- coordinate (y axis mid) (0,8);
\draw[black] (0,0) --(8,8);

%labels      
	\node[below=0.3cm] at (x axis mid) {Birth};
	\node[rotate=90, above=0.5cm] at (y axis mid) {Death};

\filldraw [color=blue!60,fill=blue!10, ultra thin, dashed] (0,2) rectangle (2,4);
\filldraw [color=white!60,fill=orange!10, ultra thin, dashed] (0,4) rectangle (2,8);
\filldraw [color=white!60,fill=red!10, ultra thin, dashed] (2,4) rectangle (4,8);
\filldraw[color=green, fill=green!10, ultra thin, dashed] (2,2) -- (4,4) -- (2,4) -- cycle;

\draw[color=red!60, ultra thin, dashed] (4,4) -- (4,8);
\draw[color=red!60, ultra thin, dashed] (2,4) -- (2,8);
\draw[color=red!60, ultra thin, dashed] (2,4) -- (4,4);
\draw[color=orange!60, ultra thin, dashed] (2,4) -- (2,8);
\draw[color=orange!60, ultra thin, dashed] (0,4) -- (0,8);
\draw[color=orange!60, ultra thin, dashed] (0,4) -- (2,4);

\filldraw[black] (1,0) circle (1pt) node[anchor=north] {\small $\tau_{i-1}$};
\filldraw[black] (0,1) circle (1pt) node[anchor=east] {\small $\tau_{i-1}$};
\filldraw[black] (2,0) circle (1pt) node[anchor=north] {\small $\tau_{i}$};
\filldraw[black] (0,2) circle (1pt) node[anchor=east] {\small $\tau_{i}$};
\filldraw[black] (1,2) circle (1pt) node[anchor=south] {\small ($\tau_{i-1}, \tau_{i})$};

\filldraw[black] (8,8) circle (0.1pt) node[anchor = west] {\large $\Delta$};

\end{tikzpicture}

\endminipage\hfill

\caption{The shaded region, $D_i$, of the persistence diagram corresponding to the filtration interval $[\tau_{i+1}, \tau_i]$ on a barcode for four different cases. Here $\Delta$ is the diagonal.} \label{fig:shadedregion1}
\end{figure}
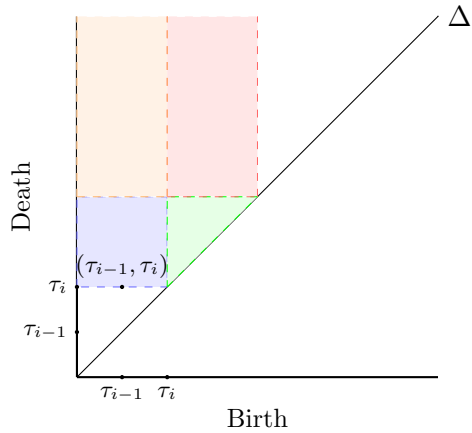

More precisely, a bar on a barcode exists at some filtration value $\tau$ in the subinterval $\left( \tau_{i -1 }, \tau_{i}\right)$ if and only if its birth-death point on the persistence diagram falls in one of the following four subregions of $D_i$:

\begin{itemize}
\item If the birth-death point lies in the red region, then the bar begins after $\tau_{i - 1}$ and ends after $\tau_{i}$. 

\item If the birth-death point lies in the orange region, then the bar begins before $\tau_{i - 1}$ and ends after $\tau_{i}$. 

\item If the birth-death point lies in the green region, then the bar begins after $\tau_{i - 1}$ and ends before $\tau_{i}$. 

\item If the birth-death point lies in the blue region, then the bar begins before $\tau_{i - 1}$ and ends before $\tau_{i}$. 

\end{itemize}

Let us now make a precise description of the  $D_i$ as a subset of $\mathbb{R}^2$ (with multiplicity). Let $D_i = K_i \setminus C_i$ where
%
%\begingroup
%\small
$$
K_{i} = \{ (b,d) \in \mathbb{R}^2 \, \vert \quad 0 < b < \tau_{i},  d > \tau_{i - 1} \} \hspace{0.1in}
\hspace{0.2 in }
$$
and 
$$% \text{ and } \hspace{0.2in}
C_{i} = \{ (b,d) \in \mathbb{R}^2 \, \vert \quad \tau_{i - 1} \leq b \leq \tau_{i},  \tau_{i -1} < d < b \}
$$

%\endgroup

\begin{definition} Let $B$ be a persistence diagram with finitely many off-diagonal points. The Betti sequence of $B$ is defined as $$\vec{v}\,(B) = \left(v_i\right)_{1}^N \in \mathbb{R}^N$$ where $v_i =| D_i \bigcap B|$ is the cardinality of the intersection of $D_i$, described above, and the persistence diagram $B$. 
\end{definition}

\subsection{Instability of the Betti Sequence}

Recall the definition of the $p$-Wasserstein distance between persistence diagrams \cite{wasserstein:1969}.

\begin{definition}\label{wasserstein} The $p$-Wasserstein distance between persistence diagrams $X$ and $Y$ is
\begingroup
\small
$$W_p(X,Y) = \left[ \inf_{\eta:X \to Y} \sum_{x \in X} || x - \eta(x)||_{\infty}^p \right]^{1/p}$$
\endgroup
where $\eta:X \to Y$ is a partial matching of $X$ and $Y$. Note as $p \to \infty$ this distance becomes the {\it bottleneck distance}. 
\end{definition}

\begin{theorem}
Let $B$ and $B'$ be persistence diagrams containing only finitely many off-diagonal points. The Betti sequence is \textbf{not stable} with respect to the 1-Wasserstein distance. That is, there exists persistence diagrams $B$ and $B'$ such that there does not exist a non-negative constant $C$ such that 
$$ || \vec{v}\,(B) - \vec{v}\,(B') ||_{\infty} \leq C \cdot W_1(B,B') $$
\end{theorem}

\begin{proof}

We will prove by example. Let $B$ be a persistence diagram with finitely many off-diagonal points and with maximum filtration $\tau_{max}$. Suppose further, for diagram $B$, that the number of subintervals, $N$, of $[0, \tau_{max}]$ is fixed so that there exists exactly one birth-death point $x_i = (b_i,d_i)$ with $\tau_{i-1} < b_i,d_i < \tau_i$ in each of the non-overlapping parts of the regions $D_i$, described above, for $1 \leq i \leq N$ as seen in Figure \ref{fig:proofdiagram}. Fix an index $j$ and for any $\epsilon > 0$ let $d_j = \tau_{j} - \frac{\epsilon}{2}$. The Betti sequence vector, $\vec{v} \,(B)$,  of this persistence diagram is then, by definition, given by the following 
$$\vec{v} \,(B) = \langle 1,...,1,1,1,...,1 \rangle^{T}$$

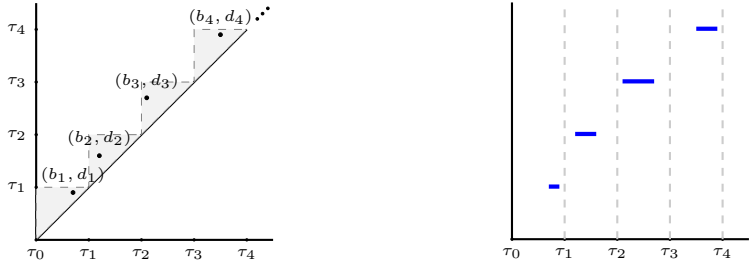
\begin{figure}[!htb]

\minipage{0.49\textwidth}

\centering
\begin{tikzpicture}[scale=0.7]
\draw[black, thick] (0,0) -- (4.5,0);
\draw[black,  thick] (0,0) -- (0,4.5);
\draw[black,  thick] (0,0) -- (4,4);

\filldraw[color=gray, fill=gray!10, ultra thin, dashed] (0,0) -- (1,1) -- (0,1) -- cycle;
\filldraw[color=gray, fill=gray!10, ultra thin, dashed] (1,1) -- (2,2) -- (1,2) -- cycle;
\filldraw[color=gray, fill=gray!10, ultra thin, dashed] (2,2) -- (3,3) -- (2,3) -- cycle;
\filldraw[color=gray, fill=gray!10, ultra thin, dashed] (3,3) -- (4,4) -- (3,4) -- cycle;

\filldraw[black] (0.7,0.9) circle (1pt) node[anchor=south] {\tiny $(b_1,d_1)$};
\filldraw[black] (1.2,1.6) circle (1pt) node[anchor=south] {\tiny $(b_2,d_2)$};
\filldraw[black] (2.1,2.7) circle (1pt) node[anchor=south] {\tiny $(b_3,d_3)$};
\filldraw[black] (3.5,3.9) circle (1pt) node[anchor=south] {\tiny $(b_4,d_4)$};

\filldraw[black] (4.2,4.2) circle (0.7pt);
\filldraw[black] (4.3,4.3) circle (0.7pt);
\filldraw[black] (4.4,4.4) circle (0.7pt);

\filldraw[black] (0,0) circle (0.6pt) node[anchor=north] {\tiny $\tau_{0}$};
\filldraw[black] (1,0) circle (0.6pt) node[anchor=north] {\tiny $\tau_{1}$};
\filldraw[black] (2,0) circle (0.6pt) node[anchor=north] {\tiny $\tau_{2}$};
\filldraw[black] (3,0) circle (0.6pt) node[anchor=north] {\tiny $\tau_{3}$};
\filldraw[black] (4,0) circle (0.6pt) node[anchor=north] {\tiny $\tau_{4}$};

\filldraw[black] (0,1) circle (0.6pt) node[anchor=east] {\tiny $\tau_{1}$};
\filldraw[black] (0,2) circle (0.6pt) node[anchor=east] {\tiny $\tau_{2}$};
\filldraw[black] (0,3) circle (0.6pt) node[anchor=east] {\tiny $\tau_{3}$};
\filldraw[black] (0,4) circle (0.6pt) node[anchor=east] {\tiny $\tau_{4}$};

\end{tikzpicture}

\endminipage\hfill
\minipage{0.49\textwidth}

\centering
\begin{tikzpicture}[scale=0.7]
\draw[black, thick] (0,0) -- (4.5,0);
\draw[black,  thick] (0,0) -- (0,4.5);

\draw[blue,  ultra thick] (0.7,1) -- (0.9,1);
\draw[blue,  ultra thick] (1.2,2) -- (1.6,2);
\draw[blue,  ultra thick] (2.1,3) -- (2.7,3);
\draw[blue,  ultra thick] (3.5,4) -- (3.9,4);

\filldraw[black] (0,0) circle (0.6pt) node[anchor=north] {\tiny $\tau_{0}$};
\filldraw[black] (1,0) circle (0.6pt) node[anchor=north] {\tiny $\tau_{1}$};
\filldraw[black] (2,0) circle (0.6pt) node[anchor=north] {\tiny $\tau_{2}$};
\filldraw[black] (3,0) circle (0.6pt) node[anchor=north] {\tiny $\tau_{3}$};
\filldraw[black] (4,0) circle (0.6pt) node[anchor=north] {\tiny $\tau_{4}$};

\draw[gray!40, thick, dashed] (1,0) -- (1,4.4);
\draw[gray!40, thick, dashed] (2,0) -- (2,4.4);
\draw[gray!40, thick, dashed] (3,0) -- (3,4.4);
\draw[gray!40, thick, dashed] (4,0) -- (4,4.4);

\end{tikzpicture}

\endminipage\hfill
\caption{Left: The persistence diagram where all persistence points exist in the non-overlapping regions of the $D_i$ (the shaded triangles). Right: The corresponding persistence barcode.} \label{fig:proofdiagram}

\end{figure}

Now consider another persistence diagram $B'$ with finitely many off-diagonal points with the same maximum filtration, $\tau_{max}$ and the same number of intervals, $N$. Suppose $B'$ is almost an exact copy of $B$ except that $d_j$ has been shifted by $\epsilon$ to become $d_j ' = d_j + \epsilon = \tau_{j} + \frac{\epsilon}{2}$. Then the persistence point $x_j = (b_j, d_j')$ is in $D_{j + 1}$ and the Betti sequence for $B'$ is
$$\vec{v}\,(B') = \langle 1, ..., 1, 2, 1, ..., 1 \rangle^{T}$$
where $ v_{j} = 1, v_{j + 1} = 2$, and $v_{j + 2} = 1$. 
For the Betti sequence vectorization to be stable with respect to the 1-Wasserstein distance under the small perturbation of $\epsilon$ we need a non-negative constant $C$ such that
$$|| \vec{v}\,(B) - \vec{v}\,(B') ||_{\infty} \leq C \cdot W_1(B,B')$$
Clearly, $|| \vec{v}\,(B) - \vec{v}\,(B') ||_{\infty} = 1$ and if we recall the definition of the 1-Wasserstein distance
$$W_1(B,B') = \inf_{\eta:B \to B'} \sum_{x \in B} ||x - \eta(x)||_{\infty}$$
where $\eta$ is a partial matching of $B$ and $B'$, we know that $W_1(B,B') = || (b_j, d_j) - (b_j, d_j') ||_{\infty} = \epsilon$.
Thus for stability we need a non-negative constant $C$ such that $1\leq C \cdot \epsilon \iff \frac{1}{\epsilon} \leq C$.
However, as $\epsilon > 0$ can be made arbitrarily small, there does not exist such a constant $C$. Therefore the Betti sequence is unstable with respect to the 1-Wasserstein distance.
\end{proof}

\begin{remark}
Similarly, it can be shown that the Betti sequence is unstable with respect to the Wasserstein distance with $p \rightarrow \infty$, i.e. with respect to the bottleneck distance. 
\end{remark}
%%%%%%%%%%%%%%%%%%%%%%%%%%%%%%%%%%%%%%%%%%%%%%%%%%%%%%%%%%%%%%%%%%%%%%
%\vspace{-3mm}
\section{Stablized Betti Sequence}
%%%%%%%%%%%%%%%%%%%%%%%%%%%%%%%%%%%%%%%%%%%%%%%%%%%%%%%%%%%%%%%%%%%%%%
We now propose a stabilized version of the Betti sequence inspired by the Gaussian smoothing techniques seen in \cite{ johnson:2020, zielinski2018persistence} and prove its stability with respect to the 1-Wasserstein distance.

\begin{definition}
Suppose $B$ is a persistence diagram with $M$ off-diagonal points $x_j = (b_j,d_j)$ and maximum filtration $\tau_{max}$. Divide the interval $[0, \tau_{max}]$ into $N$ equal subintervals of the form $[\tau_{i-1}, \tau_{i}]$ each of length $\Delta \tau$ as above. Let $\{ G_i ; \mu_i, \Sigma_i\}_{i = 1}^{N}$ be a collection of Gaussian distributions where $\mu_i$, the mean of $G_i$, is chosen to be $\langle \tau_{i-1},\tau_{i} \rangle^T$ and where $\Sigma_i$ is the covariance matrix for $G_i$. 
% $\Delta \tau \cdot I_{2\times 2}$ where by $I_{2\times2}$ we mean the $2 \times 2$ identity matrix.
% \[\Sigma_i = \begin{pmatrix} 
%0.001\Delta \tau & 0 \\
%0 &  \tau_{\text{ max}}
%\end{pmatrix} \]
Define $w_i$ to be $w_i = \frac{1}{N}$ and note that each $w_i > 0$ and $\sum_{i = 1}^{N} w_i = 1$. Then the stable Betti Sequence vector is defined by
\begingroup
\small
$$ \vec{v}^{\text{ s}} = \left( v_i^{\text{ s}} = w_i  \sum_{j =1}^{M} p_i(x_j\,|\, \mu_i,\Sigma_i)\right)_{i = 1}^{N}$$
\endgroup
where
%\begingroup
%\small
%$$
\begin{eqnarray}
p_i(x_j \, | \, \mu_i,\Sigma_i) &=& \frac{\exp\left(-\frac{1}{2}(x_j - \mu_i)^T \Sigma_i^{-1}(x_j - \mu_i)\right)}{2\pi |\Sigma_i|^{1/2}} \nonumber \\ &=& \frac{\exp\left(-\frac{1}{2\Delta \tau} ||x_j - \mu_i||_2\right)}{2 \pi \Delta \tau} \nonumber
\end{eqnarray}
%$$
%\endgroup
is the probability density function of Gaussian $G_i$ at $x_j$. Note that $|\Sigma_i|$ is the determinant of the covariance matrix, $\Sigma_i$. 
\end{definition}

\begin{remark} The choice of $\Sigma_i$ is still an open question. Ideally, we would want to use a sharp Gaussian and so $\Sigma_i$ should be defined so that $p_i(x_j \,|\, \mu_i, \Sigma_i)$ is essentially the same for every persistence point $x_j$ in $D_i$ not ``near" the boundary of $D_i$. The choice of $\Sigma_i$ should be further studied in future work. 
%We intend to explore the choice of $\Sigma_i$ further in future work.
\end{remark}

%\subsection{Stability}

\begin{theorem}\label{thm:stability}
Let $B$ be a persistence diagram of finite size and $B'$ be the persistence diagram obtained by perturbing $B$ by an arbitrary $\epsilon > 0$ such that $W_1(B,B') \leq \epsilon$. Then there exists a non-negative constant $C < \infty$ for any $\epsilon$ such that  
$$||\vec{v}^{\text{ s}}(B) - \vec{v}^{\text{ s}}(B')||_{\infty} \leq C \cdot W_1(B,B')$$
\end{theorem}

\begin{proof}
Let $M$ be the number of off-diagonal points in $B$. Let $\eta: B \to B'$ be the partial matching that realizes the 1-Wasserstein distance between $B$ and $B'$. For a fixed $i \in \{1,...,N\}$ we have
%\begingroup
%\small
%$$
%\Bigg|v_i^{\text{ s}}(B) - v_i^{\text{ s}}(B')\Bigg| = \Bigg|w_i \sum_{j = 1}^{M} \left(p_i(x_j\, | \, \mu_i, \Sigma_i) - p_i(\eta(x_j)  \, | \, \mu_i,\Sigma_i)\right)\Bigg| \leq w_i \sum_{j = 1}^{M} \Bigg|p_i(x_j \, | \, \mu_i,\Sigma_i) - p_i(\eta(x_j) \, | \, \mu_i,\Sigma_i)\Bigg|$$ 
%\endgroup
\begin{eqnarray}
\Bigg|v_i^{\text{ s}}(B) - v_i^{\text{ s}}(B')\Bigg| &=& \Bigg|w_i \sum_{j = 1}^{M} \left(p_i(x_j\, | \, \mu_i, \Sigma_i) - p_i(\eta(x_j)  \, | \, \mu_i,\Sigma_i)\right)\Bigg|  \nonumber \\
&\leq& w_i \sum_{j = 1}^{M} \Bigg|p_i(x_j \, | \, \mu_i,\Sigma_i) - p_i(\eta(x_j) \, | \, \mu_i,\Sigma_i)\Bigg|
\nonumber
\end{eqnarray}
As $p_i : \mathbb{R}^2 \to \mathbb{R}$ is continuously differentiable it is also Lipschitz continuous with Lipschitz constant $L_i$. We get

%\begingroup
%\small
%$$w_i \sum_{j = 1}^{M} \Bigg|p_i(x_j\,| \, \mu_i,\Sigma_i) - p_i(\eta(x_j)\,| \, \mu_i,\Sigma_i)\Bigg| \leq w_i \sum_{j = 1}^{M} \Bigg|L_i(x_j - \eta(x_j))\Bigg|= w_i L_i \sum_{j = 1}^{M} \Bigg|x_j - \eta(x_j)\Bigg|= w_i L_i \cdot W_1(B, B')$$
%\endgroup
\begin{eqnarray}
w_i \sum_{j = 1}^{M} \Bigg|p_i(x_j\,| \, \mu_i,\Sigma_i) - p_i(\eta(x_j)\,| \, \mu_i,\Sigma_i)\Bigg| 
&\leq& w_i \sum_{j = 1}^{M} \Bigg|L_i(x_j - \eta(x_j))\Bigg| \nonumber \\
&=& w_i L_i \sum_{j = 1}^{M} \Bigg|x_j - \eta(x_j)\Bigg| \nonumber \\
&=& w_i L_i \cdot W_1(B, B') \nonumber 
\end{eqnarray}
If we let $C = \max_{i \in [1,...,N]} w_i L_i$ we have the desired result. 
\end{proof}

\begin{example}
Returning to a simplified version of the example used to show that the Betti sequence was unstable, we will now show that Theorem \ref{thm:stability} is satisfied for this example.

Let $B$ be the persistence diagram that contains two off-diagonal points: $x_1 = \left(0.25,0.5 - \frac{\epsilon}{2}\right)$ and $ x_2 = (0.75, 0.85)$ (see Figure \ref{fig:stablebetti}) for any $0 < \epsilon$.

\begin{figure}[hbt!]
\minipage{0.49\textwidth}
\centering
\begin{tikzpicture}[scale=1]
\draw[black, thick] (0,0) -- (4.5,0);
\draw[black,  thick] (0,0) -- (0,4.5);
\draw[black,  thick] (0,0) -- (4,4);

\filldraw[color=gray, fill=gray!10, ultra thin, dashed] (0,0) --  (2,2) -- (0,2) -- cycle;
\filldraw[color=gray, fill=gray!10, ultra thin, dashed] (2,2) -- (4,4) -- (2,4) -- cycle;

\filldraw[black] (1,1.95) circle (1pt) node[anchor=south] {\tiny $\left(0.25, 0.5 - \frac{\epsilon}{2}\right)$};
\filldraw[black] (2.8,3.4) circle (1pt) node[anchor=south] {\tiny $(0.75,0.8)$};

\draw[red, thick, dashed] (0,2) --  (2,2);
\draw[red, thick, dashed] (2,2) -- (4,4);
\draw[red, thick, dashed] (4,4) -- (4,4.4);
\draw[red, thick, dashed] (0,2) -- (0,4.4);

\draw[blue, thick, dashed] (0,0) -- (2,2);
\draw[blue, thick, dashed] (2,2) -- (2,4.4);
\draw[blue, thick, dashed] (0,0) -- (0,4.4);

\filldraw[black] (4.2,4.2) circle (0.7pt);
\filldraw[black] (4.3,4.3) circle (0.7pt);
\filldraw[black] (4.4,4.4) circle (0.7pt);

\filldraw[black] (0,0) circle (0.6pt) node[anchor=north] {\tiny $\tau_{0} = 0$};

\filldraw[black] (2,0) circle (0.6pt) node[anchor=north] {\tiny $\tau_{1} = 0.5$};
\filldraw[black] (4,0) circle (0.6pt) node[anchor=north] {\tiny $\tau_{2} = 1$};

\filldraw[black] (0,2) circle (0.6pt) node[anchor=east] {\tiny $\tau_{1} = 0.5$};

\filldraw[black] (0,4) circle (0.6pt) node[anchor=east] {\tiny $\tau_{2} = 1$};

\end{tikzpicture}

\endminipage\hfill
\minipage{0.49\textwidth}
\centering
\begin{tikzpicture}[scale=1]
\draw[black, thick] (0,0) -- (4.5,0);
\draw[black,  thick] (0,0) -- (0,4.5);

\draw[blue,  ultra thick] (1,1) -- (1.8,1);
\draw[blue,  ultra thick] (2.8,2) -- (3.4,2);

\filldraw[black] (0,0) circle (0.6pt) node[anchor=north] {\tiny $\tau_{0} = 0$};

\filldraw[black] (2,0) circle (0.6pt) node[anchor=north] {\tiny $\tau_{1} = 0.5$};

\filldraw[black] (4,0) circle (0.6pt) node[anchor=north] {\tiny $\tau_{2} = 1$};

\draw[gray!40, thick, dashed] (2,0) -- (2,4.4);

\draw[gray!40, thick, dashed] (4,0) -- (4,4.4);

\end{tikzpicture}
\endminipage\hfill
\caption[Stable Betti Sequence Example Persistence Diagram $B$]{Left, the persistence diagram $B$ with two persistence points $\left(0.25, 0.5 - \frac{\epsilon}{2}\right)$ and $(0.75, 0.8)$. The dashed blue lines outline the region $D_1$ used in the definition of the Betti sequence and the red dashed lines outline the region $D_2$. Right, its corresponding persistence barcode.}\label{fig:stablebetti}
\end{figure}

The stable Betti sequence for the persistence diagram $B$, using maximum filtration 1 and two subintervals, is easily computed as the vector $\vec{v}^S(B) = \langle v^S_1, v^S_2 \rangle ^T$ where
$$ v^S_1 = \frac{1}{ 2\pi} \left[ \exp\left(- \frac{1}{16} - \frac{\epsilon^2}{4} \right) + \exp\left(- \frac{549}{400}\right) \right]$$
$$v^S_2 = \frac{1}{ 2\pi} \left[  \exp\left(- \frac{1}{16} - \frac{\epsilon^2}{4} \right) + \exp\left(- \frac{41}{400}\right) \right]$$
Now let $B'$ be the persistence diagram, pictured in Figure \ref{fig:stablebettiBprime}, that contains two off-diagonal points: $y_1 = \left(0.25,0.5 + \frac{\epsilon}{2}\right)$ and $y_2 = (0.75,0.8)$, for any $0 < \epsilon$.

\begin{figure}[hbt!]
\minipage{0.49\textwidth}
\centering

\begin{tikzpicture}[scale=1]
\draw[black, thick] (0,0) -- (4.5,0);
\draw[black,  thick] (0,0) -- (0,4.5);
\draw[black,  thick] (0,0) -- (4,4);

\filldraw[color=gray, fill=gray!10, ultra thin, dashed] (0,0) --  (2,2) -- (0,2) -- cycle;
\filldraw[color=gray, fill=gray!10, ultra thin, dashed] (2,2) -- (4,4) -- (2,4) -- cycle;

\draw[red, thick, dashed] (0,2) --  (2,2);
\draw[red, thick, dashed] (2,2) -- (4,4);
\draw[red, thick, dashed] (4,4) -- (4,4.4);
\draw[red, thick, dashed] (0,2) -- (0,4.4);

\draw[blue, thick, dashed] (0,0) -- (2,2);
\draw[blue, thick, dashed] (2,2) -- (2,4.4);
\draw[blue, thick, dashed] (0,0) -- (0,4.4);

\filldraw[black] (1,2.05) circle (1pt) node[anchor=south] {\tiny $\left(0.25,0.5 +\frac{\epsilon}{2}\right)$};
\filldraw[black] (2.8,3.4) circle (1pt) node[anchor=south] {\tiny $(0.75,0.8)$};

\filldraw[black] (4.2,4.2) circle (0.7pt);
\filldraw[black] (4.3,4.3) circle (0.7pt);
\filldraw[black] (4.4,4.4) circle (0.7pt);

\filldraw[black] (0,0) circle (0.6pt) node[anchor=north] {\tiny $\tau_{0}=0$};

\filldraw[black] (2,0) circle (0.6pt) node[anchor=north] {\tiny $\tau_{1}=0.5$};
\filldraw[black] (4,0) circle (0.6pt) node[anchor=north] {\tiny $\tau_{2}=1$};

\filldraw[black] (0,2) circle (0.6pt) node[anchor=east] {\tiny $\tau_{1}=0.5$};

\filldraw[black] (0,4) circle (0.6pt) node[anchor=east] {\tiny $\tau_{2}=1$};

\end{tikzpicture}
\endminipage\hfill
\minipage{0.49\textwidth}
\centering

\begin{tikzpicture}[scale=1]
\draw[black, thick] (0,0) -- (4.5,0);
\draw[black,  thick] (0,0) -- (0,4.5);

\draw[blue,  ultra thick] (1,1) -- (2.1,1);
\draw[blue,  ultra thick] (2.8,2) -- (3.4,2);

\filldraw[black] (0,0) circle (0.6pt) node[anchor=north] {\tiny $\tau_{0}=0$};

\filldraw[black] (2,0) circle (0.6pt) node[anchor=north] {\tiny $\tau_{1}=0.5$};

\filldraw[black] (4,0) circle (0.6pt) node[anchor=north] {\tiny $\tau_{2}=1$};

\draw[gray!40, thick, dashed] (2,0) -- (2,4.4);

\draw[gray!40, thick, dashed] (4,0) -- (4,4.4);

\end{tikzpicture}

\endminipage\hfill
\caption[Stable Betti Sequence Example Persistence Diagram $B'$]{Left, the persistence diagram $B'$ with two persistence points $\left(0.25, 0.5 + \frac{\epsilon}{2}\right)$ and $(0.75, 0.8)$. The dashed blue lines outline the region $D_1$ used in the definition of the Betti sequence and the red dashed lines outline the region $D_2$. Right, its corresponding persistence barcode.}\label{fig:stablebettiBprime}
\end{figure}
The stable Betti sequence for the persistence diagram $B'$, using maximum filtration 1 and two subintervals, is also easily computed as the vector $\vec{v}^{S} (B') = \langle v^S_1, v^S_2 \rangle ^T$ where
$$ v^S_1 = \frac{1}{ 2\pi} \left[ \exp\left(- \frac{1}{16} - \frac{\epsilon^2}{4} \right) + \exp\left(- \frac{549}{400}\right) \right]$$
$$v^S_2 = \frac{1}{ 2\pi} \left[  \exp\left(- \frac{5}{16} + \frac{\epsilon}{2} - \frac{\epsilon^2}{4} \right) + \exp\left(- \frac{41}{400}\right) \right]$$
\bigskip
Thus, we obtain 
$$
\vec{v}^S(B) - \vec{v}^S(B') = \begin{bmatrix} 0\\
\frac{1}{2\pi} \left[ \exp\left(- \frac{1}{16} - \frac{\epsilon^2}{4} \right) -\exp\left(- \frac{5}{16} + \frac{\epsilon}{2} - \frac{\epsilon^2}{4} \right)\right]
\end{bmatrix}$$

Note that for all $0 < \epsilon$, the absolute value of the first entry in the vector above is less than or equal to that of the second entry in the vector. Thus we obtain
$$
|| \vec{v} - \vec{v} \, ' ||_{\infty} = \frac{1}{2\pi} \left[ \exp\left(- \frac{1}{16} - \frac{\epsilon^2}{4} \right) -\exp\left(- \frac{5}{16} + \frac{\epsilon}{2} - \frac{\epsilon^2}{4} \right)\right]$$

\begin{figure}[hbt!]
  \centering
	\includegraphics[width=0.6\textwidth]{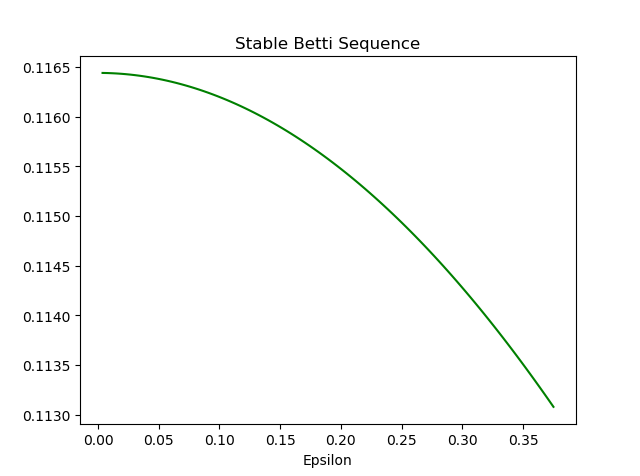}
\caption[Stable Interconnectivity Vector Example]{The plot of the ratio of the $L_{\infty}$-norm of the change in the stable Betti sequence to the 1-Wasserstein distance between persistence diagrams $B$ and $B'$.}\label{fig:stableexample}
\end{figure}
Recall, in order to satisfy Theorem \ref{thm:stability} we need
\begin{align*}
 ||\vec{v}^S(B) - \vec{v}^S(B')||_{\infty} &\leq C \cdot W_1(B, B') \\
 &= C ||\left(0.25, 0.5 - \frac{\epsilon}{2}\right) - \left(0.25, 0.5 +\frac{\epsilon}{2}\right)||_{\infty}\\ &= C \epsilon
\end{align*}
This means that we need to find a constant $C$ such that
$$ \frac{1}{\epsilon}\cdot||\vec{v}^S(B) - \vec{v}^S(B')||_{\infty} \leq C$$
Figure \ref{fig:stableexample} contains the graph of the left hand side of the inequality above versus $\epsilon$. The graph attains a maximum value of approximately $
0.116440243790144$ as $\epsilon$ goes to 0.
Thus there exists $C$, say,  $C = 0.2$, so that 
$$||\vec{v}^S(B) - \vec{v}^S(B')||_{\infty} \leq C \epsilon$$
and Theorem \ref{thm:stability} is satisfied.

\end{example}

\section{Numerical example}

For the numerical experiments, we first consider four point clouds: a uniform lattice of points with a small uniform random coordinate perturbation of magnitude at most $1/N$ on $[0,1]^2$, a uniformly random distribution on $[0,1]^2$, points drawn from the Sierpinski triangle created by the {\it chaos game} \cite{barnsley:1988} and a uniformly random distribution with a square hole of $[0.15, 0.85]^2$ on $[0,1]^2$. A sample point cloud for each type is shown in Figure \ref{clouds}. 
%Each point cloud has a total of $400$ points. 
\begin{figure}[hbt!]
  \centering
  	 \includegraphics[width=0.49\linewidth]{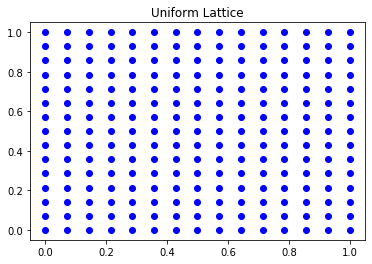}
    \includegraphics[width= 0.49\linewidth]{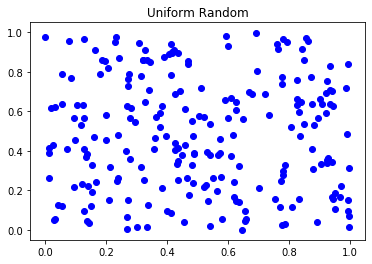}
	 \includegraphics[width=0.49\linewidth]{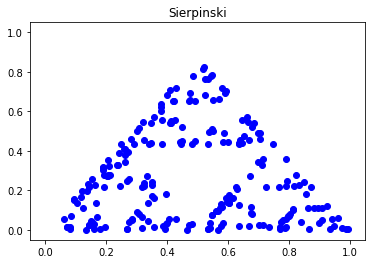}
	  \includegraphics[width=0.49\linewidth]{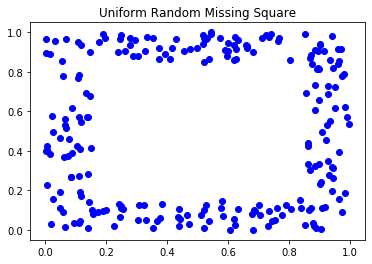}
\caption{Top: a uniform lattice of points with a small uniform random coordinate perturbation of magnitude at most $1/N$ and a uniform random distribution on $[0,1]^2$. Bottom: points drawn from the Sierpinski triangle created by the chaos game and a uniformly random distribution with a square hole of $[0.15,0.85]^2$ removed on $[0,1]^2$.}\label{clouds}
\end{figure}

Figure \ref{new1} shows a sample persistence barcode and diagram for the uniform random data (left) and Sierpinski data (right) for both zero-dimensional (red) and one-dimensional homology (blue). We observe the difference in each barcode and diagram between the uniform and Sierpinski data. Here we note that the Sierpinski data is generated with the uniform random sampling within the chaos game. 
\begin{figure}[hbt!]
\begin{center}
    \includegraphics[width=0.49\linewidth]{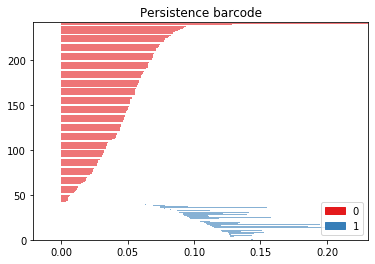}
             \includegraphics[width=0.49\linewidth]{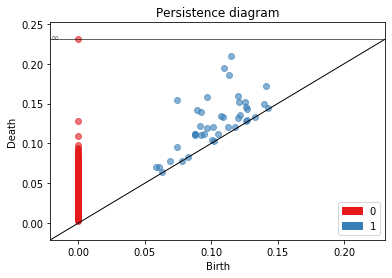}
     \includegraphics[width=0.49\linewidth]{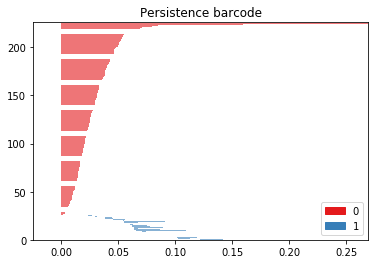}
     \includegraphics[width=0.49\linewidth]{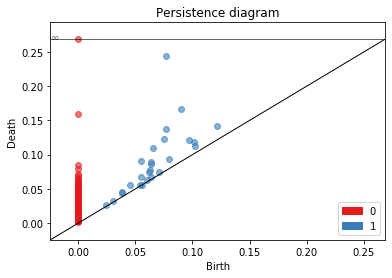}
\end{center}
\caption{Persistence barcode and diagram. Top: Uniform data. Bottom: Sierpinski data. }\label{new1}
\end{figure}
%Figure \ref{new1} shows a sample persistence barcode and diagram for the uniform random data (left) and Sierpinski data (right). We observe the difference in each barcode and diagram between the uniform and Sierpinski data. Here we note that the Sierpinski data is generated with the uniform random sampling within the chaos game. 
%\begin{figure}[hbt!]
%\begin{center}
%    \includegraphics[width=0.23\linewidth]{barcode_uniform.png}
%             \includegraphics[width=0.23\linewidth]{diagram_uniform.png}
%     \includegraphics[width=0.23\linewidth]{barcode_sierpinski.png}
%     \includegraphics[width=0.23\linewidth]{diagram_sierpinski.png}
%\end{center}
%\caption{Persistence barcode and diagram. Left:  Uniform data. Right: Sierpinski data. }\label{new1}
%\end{figure}

We will show the instability of the original Betti sequence and compare the results with the stable Betti sequence. For the stable Betti sequence, in the absence of an ideal covariance matrix $\Sigma_i$ for each Gaussian $G_i$, we use the adapted Gaussian-smoothing approach as defined below: 
%We apply the following stable scheme for the one-dimensional persistent homology: %Let ${\vec v}(B)$ be the original Betti sequence vector. 
Consider the following set, $X_i$ 
$$
   X_i =  \{ x_j \vert  \tau_{i} < b_j < \tau_i + \gamma \Delta \tau\ \}
   \cup
   \{ x_j  \vert \tau_{i-1}-\gamma \Delta \tau <  d_j <  \tau_{i-1}\}
$$
where $\gamma > 0$ is a positive constant. Since the instability we described above is induced near the domain boundary, we consider points outside $D_i$ but near the boundary of $D_i$. For the numerical example, we consider a sharp Gaussian such that all the points in $D_i \cup X_i$ participate in the Betti sequence. Since the instability is more sensitive to the lower indices of the Betti sequence, we choose the free parameter $\gamma$ as below (notice that if $\gamma = 0$, it reduces to the original Betti sequence)
$$
    \gamma = (N - i+1)/10, \quad i = 1, \cdots, N 
$$
%
%
%The stable Betti sequence described above utilizes the Gaussian smoothing applied to every point in the persistence diagram. To mimic this, we apply the Betti sequence algorithm 
%
%
%
%Instead, we apply the Gaussian smoothing to the points near the boundaries of $D_i$ only, e.g. those points in the following set, $X_i$ 
%$$
%   X_i =  \{ x_j  \vert  \tau_{i-1} \le b_j \le \tau_i + \Delta \tau,  x_j  \in D_i\ \}
%   \cup
%   \{ x_j  \vert \tau_{i-1}-\Delta \tau \le d_j \le \tau_i,  x_j  \in D_i \}
%$$
%because the instability is induced by those points in the domain boundaries of $D_i$. Following the stable scheme above for each $D_i$ with $\mu_i = (\tau_{i-1}, \tau_i)$, let ${\vec v}'(B)$ be the stable Betti sequence calculated from $X_i$.
%% Then define $v^{new}(B)$ as 
%%$$
%%     {\vec v}^{new}(B) = {\vec v}(B) +  {\vec v}'(B) 
%%$$
%%Here note that for each $D_i$, those points in $X_i$ are counted twice to build the Betti sequence vector above, one with $1$ and the other with less than $1$ as a probability. 
%%Further we define the cumulative Betti sequence, $ {\vec v}^{cum}(B)$ recursively 
%%\begin{eqnarray}
%%&&{\vec v}^{cum}_1(B) =  {\vec v}^{new}_1(B) \nonumber \\
%%&&{\vec v}^{cum}_i(B)  = {\vec v}^{new}_i(B)  + {\vec v}^{cum}_{i-1}(B), \quad i \ge 2 \nonumber
%%\end{eqnarray}
%%Then the normalized cumulative vector $ {\hat v}^{cum}(B)$ is defined as the final vector to use for classification as below  
%%$$
%% {\hat v}^{cum}(B)  =  {{ {\vec v}^{cum}(B) } \over { || {\vec v}^{cum}(B) ||_\infty}}
%%$$
Let ${\tilde D}_i$ be $X_i \cup D_i$ and $|{\tilde D}_i|$ be the cardinality of ${\tilde D}_i$. Then define a new stable Betti sequence $v^{new}(B)$ as 
$$
    % {\vec v}^{new}(B) = {\vec v}^c(B) +  {\vec v}'(B)
      {\vec v}^{new}(B) = ( |{\tilde D}_1|,  |{\tilde D}_2|, \cdots,  |{\tilde D}_N|)^T
$$
%where $ {\vec v}^c(B)$ is the original Betti sequence constructed based on the points on $D_i \backslash X_i$. 
${\vec v}^{new}(B)$ is the original Betti sequence with each domain extended by $\gamma \Delta \tau$, which can also be viewed as the Betti sequence Gaussian-smoothed with a sharp truncation near the domain boundary. 
Further, we define the cumulative Betti sequence, $ {\vec v}^{cum}(B)$ of $ {\vec v}(B)$ recursively as
\begin{eqnarray}
&&{\vec v}^{cum}_1(B) =  {\vec v}_1(B) \nonumber \\
&&{\vec v}^{cum}_i(B)  = {\vec v}_i(B)  + {\vec v}^{cum}_{i-1}(B), \quad i \ge 2 \nonumber
\end{eqnarray}
Then the normalized cumulative vector $ {\hat v}^{cum}(B)$ is defined as  
$$
 {\hat v}^{cum}(B)  =  \frac{ {\vec v}^{cum}(B) }{ || {\vec v}^{cum}(B) ||_\infty}
$$

We first show the instability of the Betti sequence. Consider the interval, $(x,y) \in \Omega = [0, L_x] \times [0, L_y]  := [0, 1+ \epsilon] \times [0,1+ \epsilon]$, where $\epsilon \ll 1$. We consider a uniform lattice of points with the total number of bins for the Betti sequence, $N = 20$, and the total number of data points, $225$. We rescale the lattice by mapping $\Omega$ to $\Omega \times 14/N$ while keeping the filtration interval $\Delta \tau = 1/N$. For this case, the shortest lattice interval becomes the same as the filtration interval $\Delta \tau = 1/N = 1/20$ when $\epsilon = 0$ and the corresponding normalized cumulative Betti sequence becomes 
$$
   {\hat v}^{cum}(B) = <0, 1, 1, \cdots, 1>
$$
because the shortest lattice interval coincides with the domain interval, $\Delta \tau$, in the persistent diagram. However, if $\epsilon < 0$, the shortest interval becomes less than $\Delta \tau$ and so the corresponding Betti sequence is
$$
   {\hat v}^{cum}(B) = <0.5, 1, 1, \cdots, 1>
$$
Here note that if $\epsilon >0$ the birth of all the bars in the barcode remain in the second domain and the death of all the bars still remain in the second domain because the perturbation $\epsilon$ is chosen small enough. Thus if the perturbation $\epsilon$ is small enough, the first element of the Betti sequence has the value of $0$ or $0.5$ while the same element in the stable Betti sequence remains almost same under the small perturbation $\epsilon$. 
The left figure in Figure \ref{new4} shows the first element of the Betti sequence versus $\epsilon$, blue for the Betti sequence and red for the stable Betti sequence with the perturbation $ \epsilon \in (-10^{-8}, 10^{8})$ -- $100$  values of $\epsilon$ were chosen uniformly. The right figure shows the same plot for the uniform random data whose domain is also $[0, 1+\epsilon]\times[0, 1+\epsilon]$. As shown in the figure, the original Betti sequence fluctuated between $0$ and $0.5$ as expected in the left figure and it also fluctuates more than the stable Betti sequence in the right figure. 
\begin{figure}[hbt!]
\begin{center}
    \includegraphics[width=0.4\linewidth]{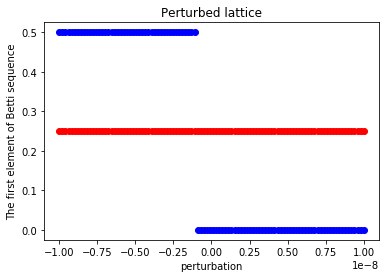}
     \includegraphics[width=0.4\linewidth]{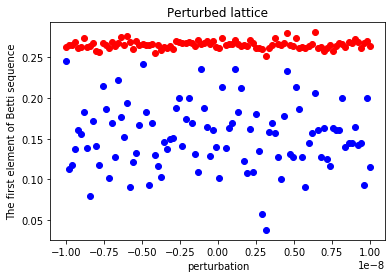}
           \end{center}
\caption{$v_1$, the first element in the sequence, versus $\epsilon$.  Blue: Betti sequence. Red: Stable Betti sequence. Left: Perturbed lattice with $\epsilon$. Right: Uniform random data.}\label{new4}
\end{figure}
Now in Figure \ref{new5} we show both original Betti sequence (blue) and stable Betti sequence (red) for the four cases shown in Figure \ref{clouds} with the fixed domain size $[0, 1]\times[0,1]$ and $N = 20$. 

 We use $100$ samples for each case and plot all the Betti sequences. The figure shows that the stable Betti sequences yield more sharp patterns while maintaining similar vector structure overall. In addition, in each data set the stable Betti sequences are more homogeneous than the Betti sequences which is promising with regards to possible machine learning applications.
\begin{figure}[hbt!]
\begin{center}
    \includegraphics[width=0.49\linewidth]{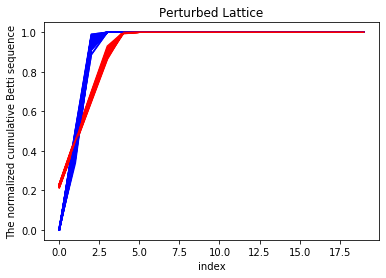}
     \includegraphics[width=0.49\linewidth]{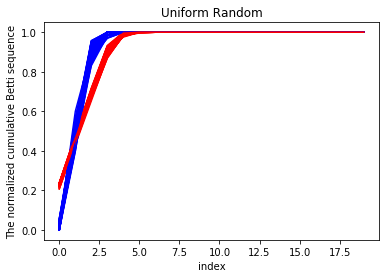}
     \includegraphics[width=0.49\linewidth]{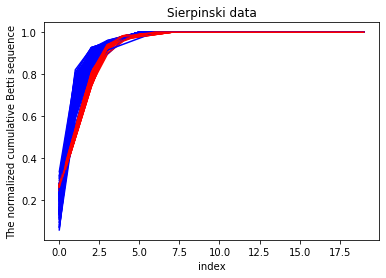}
     \includegraphics[width=0.49\linewidth]{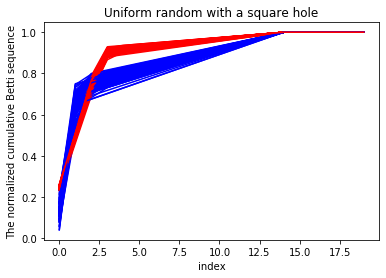}
           \end{center}
           \vskip -.2in
\caption{The Betti sequence (blue) and stable Betti sequence (red) with fixed domain size, $[0,1]\times[0,1]$. Top: a lattice with a small perturbation and a uniform random data. Bottom: a Sierpinski data and a uniform random with a square hole. Each point sequence was computed with $N = 20$ and there were $225$ points in each data set.}\label{new5}
\end{figure}

\section{Concluding Remarks}
Topological data analysis and its main tool, persistent homology, has recently gained attention in the scientific community and has proven to be highly useful in various applications. Recently, sizable research has been conducted to combine topological data analysis and machine learning. However, the representations of persistent homology, the persistence diagram and barcode, in their raw forms are not suitable to incorporate into a machine learning workflow and proper feature maps are necessary including vectorization methods. 
In this paper, we considered the Betti sequence, as a vectorization and showed by example its instability with respect to the 1-Wasserstein distance. In addition, we proposed a stable Betti sequence and proved its stability. With numerical examples, we devised a cumulative stable Betti sequence and showed that the stable Betti sequence was able to achieve a faithful representation of the Betti sequence in that it performs better with the smaller number of the filtration intervals at distinguishing data sets. 
%
% In addition, the number of subintervals needed to achieve this representation of the Betti sequence seemed to vary depending on the underlying data set. 
Our future research will incorporate the proposed stable Betti sequence into machine learning algorithms to study its effectiveness in various applications.

%
%\section{Equations And Theorems}
%
%\subsection{Equations}
%Equations should be centered. Please use Arabic numbers for
%numbering of equations. An example is as follows:
%
%\begin{equation}
%\begin{aligned}
%\frac{d}{dt}\Bigl[x(t)-g(t,x_t)\Bigr]&= Ax(t)+Bu(t)
%   +f\Bigl(t,x_t,\int_0^t q(t,s,x_s)ds\Bigr), \ t \in[0,b]=J,\\
%x_0&= \phi\in {\mathcal B}.
%\end{aligned}
%\end{equation}
%
%
%\subsection{Theorems}
%Here is an example for theorems.
%\begin{theorem}
%   This is a test theorem format.
%\end{theorem}
%\begin{proof}
%   The proof of Theorem 1 is put here.
%\end{proof}
%Here goes a corollary of the theorem.
%\begin{corollary}
%   This is a corollary.
%\end{corollary}
%
%
%\section{Figures And Tables}
%Figures and tables should be centered. Both figure and table
%captions should be Times New Roman, and enumerated using Arabic
%numbers. Examples are as follows:
%
%\begin{figure}[h]
% \begin{center}
% \includegraphics[width=2.5in, height=2.5in]{ksiam.eps}
% \caption{A sample figure caption.}
% \end{center}
%\end{figure}
%
%\begin{table}[h]
%\begin{center}
%\caption{A sample table caption.}
%\begin{tabular}{|c|c|c|c|c|} \hline
%Data 1 & Data 2 & Data 3 & Data 4 & Data 5 \\ \hline
%Value 1 & Value 2 & Value 3 & Value 4 & Value 5 \\ \hline
%\end{tabular}
%\end{center}
%\end{table}

\section*{Acknowledgments}
MJ was funded, in part, by the Doctoral Dissertation Fellowship of the Department of Mathematics at the University at Buffalo. JHJ  has been supported by Samsung Science \& Technology Foundation under grant number SSTF-BA1802-02.

\bibliography{References}{}
\bibliographystyle{plain}

%\begin{thebibliography}{99}
%
%\bibitem{Xie} X. Xie and M. Mirmehdi,
%{\it RAGS: Region-Aided Geometric Snake},
%IEEE Transactions on Image Processing, {\bf 13} (2004), 640--652.
%(Reference format for journal papers)
%
%\bibitem{Gil} D. Gil and P. Radeva,
%{\it Curvature vector flow to assure convergent
%deformable models for shape modeling},
%Lecture Notes in Computer Science, Springer Verlag,
%Proceedings of EMMCVPR, Lisbon, Portugal 2003.
%(Reference format for conference proceeding papers)
%
%\bibitem{Aubert} G. Aubert and P. Kornprobst,
%{\it Mathematical problems in image processing},
%Applied mathematical sciences 147, Springer Verlag, New York, 2002.
%(Reference format for books)

%\end{thebibliography}

\end{document}